\documentclass[12pt]{article}

\usepackage{amsthm}
\usepackage{amssymb}
\usepackage{amsmath}
\usepackage{epsfig}
\usepackage{psfrag, graphics}
\usepackage{amscd}
\usepackage{graphicx}
\usepackage{epstopdf}

\usepackage{graphicx}
\usepackage{epstopdf}

\newtheorem{theorem}{Theorem}[section]
\newtheorem{corollary}[theorem]{Corollary}
\newtheorem{proposition}[theorem]{Proposition}
\newtheorem{lemma}[theorem]{Lemma}

\theoremstyle{definition}

\newtheorem{remark}[theorem]{Remark}

\def\Q{\mathbb Q}

\def\PP{\mathcal P}

\newcommand{\CC}{{\mathbb C}}

\newcommand{\NN}{{\mathbb N}}
\newcommand{\RR}{{\mathbb R}}
\newcommand{\ZZ}{{\mathbb Z}}

\begin{document}

\title{Ellipsoid embeddings and symplectic packing stability}

\author{O. Buse \and R. Hind}

\date{\today}

\maketitle

\begin{abstract}
We prove packing stability for any closed symplectic manifold with rational cohomology class. This will rely on a general symplectic embedding result for ellipsoids which assumes only that there is no volume obstruction and that the domain is sufficiently thin relative to the target. We also obtain easily computable bounds for the Embedded Contact Homology capacities which are sufficient to imply the existence of some volume preserving embeddings in dimension $4$.
\end{abstract}

\section{Introduction}

We will study symplectic embeddings, that is, embeddings $f:(N,\sigma) \to (M, \omega)$ between symplectic manifolds such that $f^* \omega = \sigma$. As our domains will be ellipsoids or unions of balls in Euclidean space it is sometimes convenient to introduce the weaker notion $(N,\sigma) \hookrightarrow (M,\omega)$ to mean that $\lambda(N,\sigma) := (N,\lambda \sigma) \to (M,\omega)$ for all $0 < \lambda <1$. We note that when $M$ is a ball and $N$ is an ellipsoid or a disjoint union of balls in dimension $4$, then $(N,\sigma) \hookrightarrow (M,\omega)$ implies there exists a symplectic embedding $(\mathring{N},\sigma) \to (M,\omega)$ from the interior $\mathring{N}$ of $N$, see \cite{mcduff0}, \cite{mcduff1}, Corollary $1.6$. However no such results are known in higher dimension. In any case, we say that $(N,\sigma)$ {\it fully fills} $(M,\omega)$ if $\mathrm{vol}(N,\sigma)=\mathrm{vol}(M,\omega)$ and $(N,\sigma) \hookrightarrow (M,\omega)$.

Our main theorem is an embedding result for ellipsoids. To state this, we denote by $E(a_1, \dots, a_n) \subset \RR^{2n}$ the ellipsoid \[E(a_1, \dots, a_n)=\left\{\sum_{i=1}^n \frac{\pi(x_i^2+ y_i^2)}{a_i} \le 1\right\}.\]
Ellipsoids inherit a symplectic structure from the standard form $\omega_0 = \sum_{i=1}^n dx_i \wedge dy_i$ on $\RR^{2n}$. In our notation, the ball of capacity $c$ is written $B^{2n}(c)=E(c, \dots, c)$ and $\lambda E(a_1, \dots, a_n)= E(\lambda a_1, \dots, \lambda a_n)$. Any given ellipsoid is  symplectomorphic to the same ellipsoid with its factors permuted. Unless otherwise stated, when describing an ellipsoid we will list the factors in increasing order.

\begin{theorem} \label{main} There exists a constant $S(b_1, \dots ,b_n)$ such that if $\frac{a_n}{a_1}>S$ and $a_1 \dots a_n \le b_1 \dots b_n$ there exists an embedding $$E(a_1, \dots ,a_n) \hookrightarrow E(b_1, \dots ,b_n).$$
\end{theorem}

This will be established using a technique from Buse-Hind \cite{busehind} which generates higher dimensional embeddings from lower dimensional ones, together with some new ellipsoid embeddings in dimension $4$.

Combining with work of Biran \cite{bir0} and Opshtein \cite{ops2}, we also have a full filling result valid for general rational symplectic manifolds. Recall that the manifold $(M,\omega)$ is {\it rational} if $[\omega] \in H^2(M,\Q) \subset H^2(M,\RR)$.

\begin{theorem}\label{general2}
If $(M^{2n},\omega)$ (or $(M^{2n},\lambda \omega)$ for any $\lambda>0$) is rational, then  there exist a constant $S(M)>0$ such that for all ellipsoids $E(a_1, \dots ,a_n) $ with $\frac{a_n}{a_1}>S(M)$ we have a volume filling embedding  $E(a_1, \dots ,a_n) \hookrightarrow (M,\omega).$
\end{theorem}

In the paper \cite{hoferschlenklatsch}, section $3$, Cieliebak, Hofer, Schlenk and Latschev define generalized symplectic capacities defined in terms of a fixed symplectic manifold $(M,\omega)$. If $M$ has dimension $2n$ then we get a capacity on the space of $2n$ dimensional ellipsoids defined by

\[c^{(M,\omega)}E(a_1, \cdots, a_n):= \inf \{ c | E(a_1, \cdots, a_n) \hookrightarrow (M, c \omega) \}.\]

Our result can be interpreted as saying that $c^{(M,\omega)}E(a_1, \cdots, a_n)$ coincides with the normalized volume capacity

\[ v^{(M,\omega)}E(a_1, \cdots, a_n):= \sqrt[n]{\frac {{\rm vol}(E(a_1 \cdot\ldots\cdot a_n))}{{\rm vol}(M)}}\]
whenever $M$ is rational and $\frac{a_n}{a_1}>S(M)$, that is, whenever the ellipsoid is sufficiently thin.

The existence of an embedding between $4$-dimensional ellipsoids is determined by the Embedded Contact Homology (ECH) capacities of M. Hutchings, see \cite{hutch}. Namely, there is a sequence of numbers $\mathcal{N}(a,b)(k)$, $n \ge 0$, associated to a pair $(a,b)$ and according to D. McDuff, \cite{mcduff2}, we have $E(a_1,a_2) \hookrightarrow E(b_1,b_2)$ if and only if $\mathcal{N}(a_1,a_2)(k) \le \mathcal{N}(b_1,b_2)(k)$ for all $k \ge 0$. Although in principle this solves the problem of ellipsoid embeddings in dimension $4$, it remains a difficult question to determine if a particular embedding exists since the $\mathcal{N}(a,b)(k)$ are given by quite complicated combinatorial formulas, \cite{hutch}. We will give quadratic bounds on a related piecewise linear function, see Proposition \ref{theparabolas}. One consequence is the following.

\begin{theorem} \label{oneemb}
Let $\beta \geq 1$ and $\alpha \geq \frac{(5 \beta+16)^2}{16 \beta}$. Then there exists a full filling $E(1,\alpha)  \hookrightarrow \sqrt{\frac{\alpha}{\beta}} E(1,\beta)$.
\end{theorem}



Finally we discuss packing stability. The
$k^{\text{th}}$ packing number of a compact, $2n$-dimensional, symplectic
manifold $(M, \omega)$ is
\[
  p_k (M,\omega) =
  \frac{\sup_c {\rm Vol}(\sqcup_k B(c))}{{\rm Vol}(M,\omega)},
\]
where the supremum is taken over all $c$ for which there exist a
symplectic embedding of $\sqcup_k B(c)$, the disjoint union of $k$ balls of capacity $c$, into $(M,\omega)$. Naturally,
$p_k(M,\omega) \leq 1$. The identity $p_k(M,\omega)=1$ is equivalent to saying that
$(M,\omega)$ admits a full filling by $k$ identical balls, otherwise we say that
there is a packing obstruction.

The symplectic manifold $(M,\omega)$ has {\it packing stability} if there exists an integer $N_{{\rm stab}}(M,\omega)$ such that $p_i(M,\omega)
= 1$ for all $i \geq N_{{\rm stab}}(M,\omega)$.

\begin{theorem} \label{stable} If $(M^{2n},\omega)$ (or $(M^{2n},\lambda \omega)$ for any $\lambda>0$) is rational, then it has packing stability.
\end{theorem}

This is due to P. Biran, \cite{Bi2}, in the case that $M$ is a closed four dimensional symplectic manifold with a rational cohomology class, and to Buse-Hind \cite{busehind} when $M$ is $\CC P^n$ with its Fubini-Study symplectic form.

Bounds on $N_{{\rm stab}}(M,\omega)$ can be derived from Theorem \ref{main}, although in specific examples we can obtain much sharper estimates.

Let $\CC P^n$ be equipped with its Fubini-Study symplectic form, and denote by $H^n_d$ a smooth hypersurface of degree $d$ in $\CC P^{n+1}$ with the induced symplectic form.

\begin{theorem} \label{bound}
(i) $N_{{\rm stab}}(\CC P^n) \le \lceil (8\frac{1}{36})^{\frac{n}{2}} \rceil$. \newline
(ii) $N_{{\rm stab}}(H^n_d) \le \lceil (\frac{25}{16}d+10d^{-\frac{(n-2)}{n}}+16d^{-\frac{2(n-1)}{n}})^{\frac{n}{2}} \rceil$.
\end{theorem}

Of course, statement $(i)$ here is just a refined estimate of statement $(ii)$ valid for hyperplanes and it improves the previous bounds provided in \cite{busehind}.

One can also consider packing with other symplectic manifolds, typically replacing the ball by other open subsets of $(\RR ^{2n}, \omega_0)$. Fixing a symplectic $(D,\omega_0)$ of dimension $2n$ we define the $k^{\text{th}}$ packing number with respect to $D$ of a $2n$-dimensional symplectic
manifold $(M, \omega)$ by
\[
  p^D_k (M,\omega) =
  \frac{\sup_c {\rm Vol}(\sqcup_k (D,c\omega_0))}{{\rm Vol}(M,\omega)},
\]
where the supremum is taken over all $c$ for which there exist a
symplectic embedding of $\sqcup_k (D,c\omega_0)$ into $(M,\omega)$. Similarly to the above, we say that $(M,\omega)$ has packing stability with respect to $(D,\omega_0)$ if $p^D_k(M,\omega)=1$ for all $k$ sufficiently large. Taking $D$ to be an ellipsoid, we can generalize Theorem \ref{stable} as follows.

\begin{theorem} \label{stable2} If $(M^{2n},\omega)$ (or $(M^{2n},\lambda \omega)$ for any $\lambda>0$) is rational, then it has packing stability with respect to any symplectic ellipsoid.
\end{theorem}

We know of no manifolds which have packing stability with respect to other domains in $\RR^{2n}$, for example polydisks. 

{\bf Outline of the paper.}

In section $2$ we discuss the ECH capacities, deriving our quadratic bounds and in particular proving Theorem \ref{oneemb}. Section $3$ gives the proof of Theorem \ref{main}. In section $4$ we describe how to obtain Theorem \ref{general2} and packing stability, Theorems \ref{stable} and \ref{stable2}, from Theorem \ref{main} together with a construction of E. Opshtein, \cite{opshtein}. In section $5$ we compute some examples and prove Theorem \ref{bound}.

{\bf Acknowledgements.}

The first author would like to thank Dusa McDuff for suggesting Opshtein's paper as a potential starting point towards adressing the packing problems in higher dimension, and  Michal Misiurewicz and Rodrigo P{\'e}rez for several fruitful discussions.
\section{ Ellipsoid embeddings in dimension four}

In this section we consider the embedding capacity function

\[f_{\beta}(\alpha) := \inf\{c|E(1,\alpha) \hookrightarrow c E(1,\beta)\}\]
which is a natural extension of the function $f_1$ considered by McDuff and Schlenk in the  paper \cite{mcdsch}, see Theorem $1.1.2$.

We will show that $f_{\beta}(\alpha)$ coincides with the normalized volume capacity of the ellipsoid $E(1,\beta)$ for sufficiently large values of $\alpha$.

\subsection{ Description of ECH  and some new estimates}

The key ingredients for our study are results of Hutchings \cite{hutch} and McDuff \cite{mcduff2} on Embedded Contact Homology which together give necessary and sufficient conditions for a four dimensional ellipsoid embedding.

A four dimensional Liouville domain is a compact exact symplectic manifold $(X, \omega)$ such that there exists a contact form on $\partial X$ which is a primitive of $\omega|_{\partial X}$. Hutchings associates to such $(X,\omega)$ an increasing sequence of real numbers $c_k(X,\omega)$ for $k \ge 0$ called the ECH capacities. The term capacity is justified by the following.


\begin{theorem}(Hutchings \cite{hutch}) \label{echemb}
Let $(X, \omega)$  and $(X', \omega')$ be two Liouville domains as above.
If there exist a symplectic embedding $\phi: (X, \omega) \longrightarrow ( \mathring{X}', \omega')$ then $c_k(X, \omega) \leq c_k(X', \omega')$ for all integers $k>0$.
\end{theorem}
In the same paper, Hutchings proceeds to compute the capacities of ellipsoids $E(a,b)$:
\begin{proposition}
 Given $0 <a \leq b$ consider  the sequence
    $\mathcal { N}(a,b)(k)$  obtained by arranging in increasing order,
    with repetitions, all the numbers of the type $a \ell+b p$ with $\ell,p$
    any nonnegative integers. Then $c_k(E(a,b))=  \mathcal { N}(a,b)(k-1) $

    \end{proposition}

Following this, McDuff showed that the necessary condition for ellipsoid embeddings coming from Theorem \ref{echemb} is also sufficient:

\begin{theorem} (McDuff \cite{mcduff2})\label{mcduffECH} There exists a symplectic embedding
$\;\;\mathring{E}(a,b) \longrightarrow E(a',b')$ if and only if  $\mathcal{N}(a,b)(k) \leq \mathcal{N}(a',b')(k)$ for all natural numbers $k$.
\end{theorem}

Comparing Theorems \ref{echemb} and \ref{mcduffECH} we remark that $\mathring{E}(a,b) \to E(a',b')$ if and only if $E(a,b) \to \mathring{E}(\lambda a',\lambda b')$ for all $\lambda>1$, see \cite{mcduff2} Remark 1.3.

We will refer to $\mathcal{N} (a,b)(x)$ as being the piecewise linear function built by joining by line segments  the points of $\mathcal{N}(a,b)(k)$. It is clearly sufficient to compare two such piecewise functions and moreover, if one defines, for any $y \geq 0$

\begin{equation}
R(a,b)(y)= \sup\{k | \mathcal{N}(a,b)(k) \leq y\} \end{equation}
then we have
\begin{corollary}\label{corhut}(see Hutchings \cite{hutch}, Bauer \cite{bauer}) $E(a,b) \hookrightarrow E(a',b')$
if and only if
\begin{equation}
 \label{rineq}
  R(a,b)(y) \geq R(a',b')(y)
\end{equation}
 for all $y>0$.
\end{corollary}

To prove theorem \ref{oneemb} by using this corollary, we will first need the following estimates:
\begin{proposition}\label{theparabolas}
For any $y >0$, and $a<b$ we have that

\begin{equation}\label{rineqparab}
  \frac{y^2}{2ab} + \frac{y}{2a} \leq
  R_{a,b}(y) \leq
  \frac{y^2}{2ab} + \frac{y}{2a} + \frac{y}{b} + \frac{b}{8a} + 1.
\end{equation}

\end{proposition}

\begin{proof}
As pointed in \cite{hutch}, $R_{a,b}(y)$ is interpreted geometrically as the number
of non-negative integer vectors $(m,n)$ in the closed triangle bounded by
$m=0$, $n=0$, and the diagonal $am+bn=y$. Equivalently, this represents the
area of the union of unit squares with lower-left corners $(m,n)$ such that
$am+bn \leq y$. The triangular area $\frac{y^2}{2ab}$ under the diagonal is
then an obvious lower estimate for $R_{a,b}(y )$, but it omits the sizable
area of the staircase-shaped region $S$ above the diagonal. Let $\{y/b\}=\frac{y}{b} - \lfloor y/b \rfloor$ denote the fractional part of $\frac{y}{b}$.

\begin{figure}[h]\begin{center}
  \includegraphics{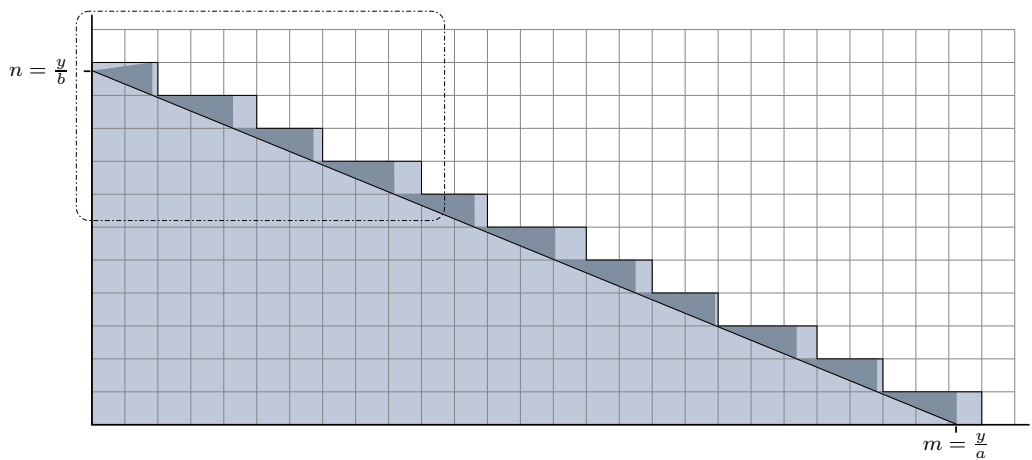}
  \caption{Triangles $T_k$ and the region $S$}
  \label{fig:Pack1}
\end{center}\end{figure}
To estimate the area of $S$, let $p_0 = (0,y/b)$, and let $p_1, \ldots,
p_{\lceil y /b \rceil}$ be the points of integer height along the diagonal,
indexed from left to right. For every $k$ with $1 \le k \le \lceil y/b \rceil$,
let $T_k$ be the triangle formed by $p_{k-1}$, $p_k$, and the point one unit
above $p_k$. Every triangle $T_k \subset S$ rests atop the diagonal, these are the darkest triangles in Figure \ref{fig:Pack1}. They all
have (vertical) base of length 1, and their (horizontal) heights add up to
$y/a$; thus, the total area of the triangles is $y/2a$, and
\[\frac{y^2}{2ab} + \frac{y}{2a} \leq R_{a,b}(y).\]

For the upper estimate, add unit squares to the right of $T_1, \ldots ,
T_{\lceil y/b \rceil}$, see Figure \ref{fig:Pack2}. These have total area $\lceil y/b \rceil$, and cover the remainder of
$S$, except (sometimes) for a thin triangle above $T_1$ sitting at the top left of Figure \ref{fig:Pack2}.
The (vertical) base of this triangle has length $1 - \{y/b\}$. Its (horizontal) height
is the first coordinate of $p_1$, that is, $\frac{b}{a}\{y/b\}$. Therefore we get
\[
  R_{a,b}(y) \leq
  \frac{y^2}{2ab} + \frac{y}{2a} + \lceil \frac{y}{b} \rceil + \frac{b}{2a}\{y/b\}(1 - \{y/b\})
 \le  \frac{y^2}{2ab} + \frac{y}{2a} + \frac{y}{b} + \frac{b}{8a}+1.\]

where for the final inequality we note that $\{y/b\}(1-\{y/b\}) \le \frac{1}{4}$ since $0\le\{y/b\}<1$.

\begin{figure}[h]\begin{center}
  \includegraphics{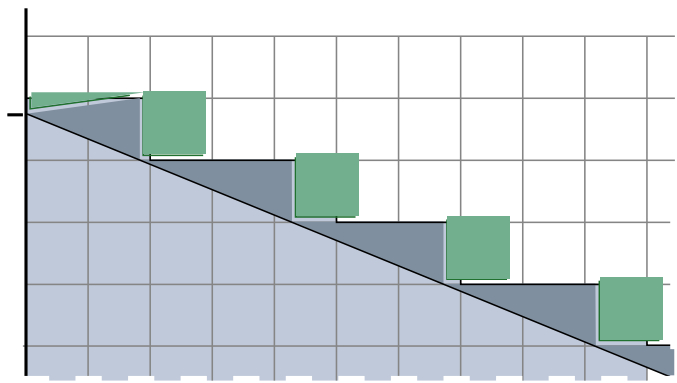}
  \caption{Added squares covering $S$}
  \label{fig:Pack2}
\end{center}\end{figure}

\end{proof}


\subsection{A volume filing ellipsoid embedding in dimension four}

We can now give a proof of theorem \ref{oneemb}:

\begin{proof}

Let us call $c:=\sqrt{\frac{\alpha}{\beta}}>1$ by hypothesis. Assuming $\alpha \geq \frac{(5 \beta+16)^2}{16 \beta}$ and $\beta>1$ it is required to show that we have a volume filling embedding $E(1,\alpha)  \hookrightarrow  E(c,c \beta)$. By Corollary \ref{corhut} this is equivalent to showing
\begin{equation}\label{rineqhere}
R_{1,\alpha}(y) \geq R_{c,c \beta}(y)
\end{equation}
for all $y>0$.

As $\alpha > \beta$ it is easy to see that for $0 \leq y \leq c \lfloor \beta \rfloor < \alpha $ we have  $R_{1,\alpha}(y)= \lfloor y \rfloor $ and $R_{c,cb}(y)= \lfloor \frac{y}{c} \rfloor$. Therefore
the inequality \eqref{rineqhere} holds for values of $y \leq c \lfloor \beta \rfloor $. Our hypothesis also implies that $\lfloor \beta \rfloor +1 \leq \lfloor \alpha \rfloor$ and so the graphs of $\mathcal{N}(1,\alpha)(x)$ and $\mathcal{N}(c,c\beta)(x)$ are as shown in Figure \ref{fig:fig3} for small values of $x$. Hence we can also observe directly that inequality \eqref{rineqhere} holds in the range $ c \lfloor \beta \rfloor \leq y \leq c \beta $.

\begin{figure}[h]\begin{center}
  \includegraphics{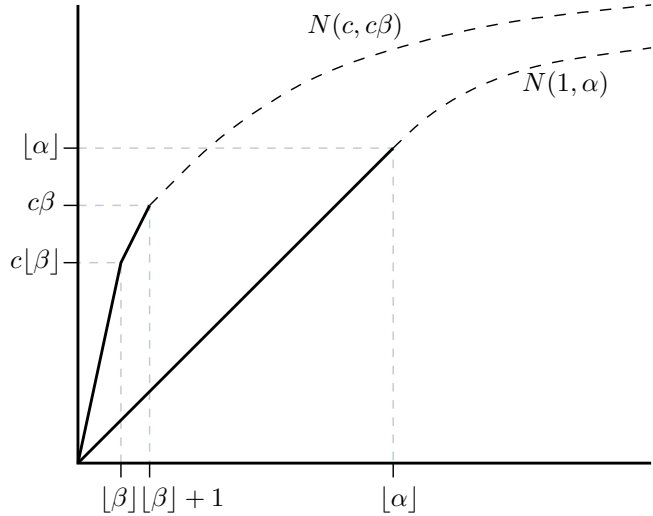}
  \caption{graphs of the ECH capacity functions}
  \label{fig:fig3}
\end{center}\end{figure}

It remains to show that \eqref{rineqhere}  holds for the remaining values $y \geq c \beta$.
Using the inequality \eqref{rineqparab} from Proposition \ref{theparabolas} twice, once for $R_{1,\alpha}(y)$ and again for  $R_{c,c \beta}(y)$,
it is sufficient to show

\begin{equation}
\frac{y^2}{2 \alpha} +\frac{y}{2} \geq \frac{y^2}{2c^2 \beta}+\frac{y}{2c}+\frac{y}{c \beta} +\frac{c \beta}{8 c} +1
\end{equation}
for $y \geq c \beta$. Using the fact that $\alpha=c^2 \beta$, this reduces to showing that

\begin{equation}
\frac{y}{2}(\frac{c\beta - \beta -2}{c\beta}) \geq \frac{\beta}{8} + 1
\end{equation}
for all $y \geq c \beta$. But this is equivalent to our hypothesis
$\alpha \geq \frac{(5 \beta+16)^2}{16 \beta}$.
\end{proof}

\section{Volume filling embeddings}

\begin{theorem}\label{general} Define a constant $S$ by $S^{\frac{1}{n-1}}=\frac{2^{n+6}}{3}20^{\frac{n-2}{2}} \max_k 20^{\frac{k(k-1)}{2}} \frac{b_1 \dots b_n}{b_k^n}$. Then if $\frac{a_n}{a_1}>S$ and $a_1 \dots a_n = b_1 \dots b_n$ there exists a volume preserving embedding $$E(a_1, \dots ,a_n) \hookrightarrow E(b_1, \dots ,b_n).$$
\end{theorem}

\
\begin{proof} We will apply a series of volume preserving embeddings to the ellipsoid $E(a_1, \dots ,a_n)$. The factors of our ellipsoids will always be in increasing order. First recall our key tools.

Let $m(x)=\frac{(5x+16)^2}{16x}$. Then $m$ is decreasing on the interval $\{1 \le x \le \frac{16}{5}\}$, increasing on $\{x \ge \frac{16}{5}\}$ and $m(\frac{16}{5})=20$. Theorem \ref{oneemb} gives the following.

\begin{theorem} \label{one} If $\frac{b}{a}>m(\frac{d}{c})$ and $ab=cd$ then $E(a,b) \hookrightarrow E(c,d)$.
\end{theorem}

\begin{corollary} \label{onecor} Let $\frac{b}{a} > \frac{2^7}{3}$ and $2a < d < \sqrt{ab}$. Then $E(a,b) \hookrightarrow E(d,\frac{ab}{d})$.
\end{corollary}

\begin{proof} The upper bound on $d$ guarantees that the factors of our image ellipsoid are in increasing order. Therefore by Theorem \ref{one} we just need to check that $\frac{b}{a}>m(\frac{ab}{d^2})$. We compute $$\frac{b}{a}-m(\frac{ab}{d^2}) = \frac{b}{a}-\frac{25}{16}\frac{ab}{d^2} - 10 - \frac{16d^2}{ab} \ge \frac{b}{a}-\frac{25}{16}\frac{b}{4a} - 26=\frac{39b}{64a}-26.$$ This last term is positive as required given our hypothesis.
\end{proof}

We will also utilize Proposition $2.1$ from Buse-Hind \cite{busehind}.

\begin{proposition} \label{two}\cite{busehind} If $E(a,b) \hookrightarrow E(c,d)$ then $$E(a,b,a_3, \dots ,a_n) \hookrightarrow E(c,d,a_3, \dots ,a_n)$$ for any $a_3, \dots ,a_n$.
\end{proposition}

Note here that as the order of the factors is irrelevant analogous statements hold for all other pairs of factors.

We start with the following lemma.

\begin{lemma} \label{first} There exists a volume preserving embedding $E(a_1, \dots ,a_n) \hookrightarrow E(a'_1, \dots ,a'_n)$ such that
\renewcommand{\theenumi}{(\roman{enumi})}
\begin{enumerate}
\item $a'_n=a_n$;
\item $\frac{a'_k}{a'_{k-1}}<20$ for all $1 \le k \le n-1$;
\item $\frac{a'_n}{a'_k}>20^{-\frac{n-2}{2}}S^{1/{n-1}}$ for all $1 \le k \le n-1.$
\end{enumerate}
\end{lemma}

\begin{proof} We perform a sequence of embeddings fixing throughout the largest factor $a_n$ of the ellipsoid.

Theorem \ref{one} implies that if $\frac{a_k}{a_{k-1}}>20$ then there exists an embedding $E(a_{k-1},a_k) \hookrightarrow \sqrt{\frac{5a_{k-1}a_k}{16}}E(1,\frac{16}{5})$.

We will apply the embedding of Proposition \ref{two} induced by this four dimensional embedding to any consecutive pair of factors $a_{k-1}, a_k$ for which $k<n$ and whose ratio is greater than $20$. After such an operation the product of the first $k-1$ factors will increase by a factor of at least $\sqrt{\frac{100}{16}}$, but the product of the first $n-1$ factors is fixed throughout. Thus after a finite number of steps we arrive at an ellipsoid $E(a'_1, \dots ,a'_n)$ satisfying condition $(ii)$.

For condition $(iii)$ we estimate
\[\left(\frac{a'_n}{a'_k}\right)^{n-1} \geq \left(\frac{a'_n}{a'_{n-1}}\right)^{n-1}\]
\[\geq \frac{(a'_n)^{n-1}}{20^{(n-1)(n-2)/2}} \cdot \frac{20 \cdot 20^2 \cdot \ldots \cdot 20^{n-2} }{(a'_{n-1})^{n-1}}\]
\[\geq \frac{(a'_n)^{n-1}}{20^{(n-1)(n-2)/2} \cdot a'_{n-1} \cdot a'_{n-2} \cdot \ldots \cdot a'_1}\]
using property $(ii)$.

Now, as all our embeddings are volume preserving and fix $a_n$, the product of the first $n-1$ terms is also preserved under our sequence of embeddings. Thus we have
\[\left(\frac{a'_n}{a'_k}\right)^{n-1} \geq \frac{a_n^{n-1}}{20^{(n-1)(n-2)/2} \cdot a_{n-1} \cdot a_{n-2} \cdot \ldots \cdot a_1}\]
\[\geq 20^{\frac{-(n-1)(n-2)}{2}}S,\]
where for the final inequality we used simply that $a_k \le a_n$ for all $k<n$ and that by hypothesis $a_n>Sa_1$. Our lemma follows.
\end{proof}

Now we drop the primes from the ellipsoid resulting from Lemma \ref{first} and write simply $E(a_1, \dots ,a_n)$ for our new range.

\begin{lemma} \label{snd} $b_k >2a_k$ for all $1 \le k \le n-1$.
\end{lemma}

\begin{proof}
For any $k$ we have
\[1=\frac{a_1 \ldots a_n}{b_1 \ldots b_n} \ge \frac{a_1 \ldots a_{k-1} a_k^{n-k}a_n}{b_1 \ldots b_n}\]
\[\geq \frac{a_k^{n-k}}{b_1 \ldots b_n} \frac{a_k}{20^{k-1}} \ldots \frac{a_k}{20} 20^{-\frac{n-2}{2}}S^{1/{n-1}} a_k\]
using properties $(ii)$ and $(iii)$ from Lemma \ref{first}
\[ = \left(\frac{a_k}{b_k}\right)^n \frac{20^{-\frac{n-2}{2}}S^{1/{n-1}}(b_k)^n}{20^{k(k-1)/2}b_1 \dots b_n} \geq \left(\frac{a_k}{b_k}\right)^n 2^n\]
and the Lemma follows.

\end{proof}

Now we complete the proof of Theorem \ref{general}. We will proceed by induction on $k=1, \dots ,n-1$. After the $kth$ stage we will have found a volume preserving embedding into an ellipsoid of the form $E(b_1,\dots ,b_k,a_{k+1},\dots ,a_{n-1},a_{n,k})$ where the $a_i$ are the factors of our latest ellipsoid above and $a_{n,k}$ will satisfy $\frac{a_{n,k}}{a_{n-1}} > \frac{M}{2^k}$. Here we define $M=20^{-\frac{n-2}{2}}S^{1/{n-1}}$, so by Lemma \ref{first} we have $\frac{a_n}{a_{n-1}} > M.$

For our inductive step, we claim that if $\frac{a_{n,k-1}}{a_{n-1}} > \frac{M}{2^{k-1}}$ then there exists an embedding $$E(a_k,a_{n,k-1}) \hookrightarrow E(b_k, \frac{a_{n,k-1}a_k}{b_k}).$$ We can then set $a_{n,k} = \frac{a_{n,k-1}a_k}{b_k}$ and provided $k<n-1$ we have $\frac{a_{n,k}}{a_{n-1}} > \frac{M}{2^k}$ by Lemma \ref{snd}. Thus by repeated applications of Proposition \ref{two} we can conclude by induction. After the $(n-1)$st step we will automatically have $a_{n,n-1}=b_n$ by the hypothesis $a_1 \dots a_n = b_1 \dots b_n$ and the fact that all of our embeddings are volume preserving.

To justify the claim, we will apply Corollary \ref{onecor} with $a=a_k$, $b=a_{n,k-1}$ and $d=b_k$. There are three conditions to check. First $$\frac{a}{b}=\frac{a_{n,k-1}}{a_{k}} \ge \frac{a_{n,k-1}}{a_{n-1}} >\frac{M}{2^{k-1}} = 20^{-\frac{n-2}{2}}2^{-k+1}S^{1/{n-1}}>\frac{2^7}{3}$$ by our assumption on $S$. Secondly, $d=b_k<2a_k$ by Lemma \ref{snd}. Finally, $d^2 = b_k^2 \le b_k b_n < a_k a_{n,k-1} = ab$ as our ellipsoids have equal volume but $b_k>a_k$ for all $k<n$.


\end{proof}

\section{Full filling by an ellipsoid}

In this section we present Opshtein's result, which is a refinement of Biran's polarization theorem,  that any rational symplectic manifold can be fully filled by an ellipsoid, that is, we prove the following.

\begin{theorem} [\cite{opshtein} \cite{ops2}]\label{sec4} Let $(M,\omega)$ be rational. Then there exist $b_1, \dots ,b_n$ such that there exists a full filling $E(b_1, \dots ,b_n) \hookrightarrow (M,\omega)$.
\end{theorem}

Given Theorem \ref{sec4} we can prove Theorems \ref{stable} and \ref{stable2}. Theorem \ref{general2} is also a direct application of Theorems \ref{sec4} and \ref{main}.

{\it Proof of Theorem \ref{stable}.} Given $(M,\omega)$ rational, we consider the fully filling ellipsoid $E(b_1, \dots ,b_n)$ given by Theorem \ref{sec4}. By Theorem \ref{main} there exists a constant $S(b_1,\dots ,b_n)$ such that there exists a full filling $ \gamma E(1,\dots ,1,k) \hookrightarrow E(b_1, \dots ,b_n)$ for any $k \ge S$. But from \cite{busehind}, Lemma $4.1$ there also exists a full filling \[\sqcup_k B(1) \hookrightarrow E\big( 1^{\times (n-1)},k \big)\] for any $k \in \NN$. Putting the two maps together, we obtain a full filling of $E(b_1, \dots ,b_n)$, and hence $(M,\omega)$, by $k$ balls of capacity $\gamma$. This means that $(M,\omega)$ has packing stability as required. $\qed$

\vspace{0.2in}

{\it Proof of Theorem \ref{stable2}.} Fix an ellipsoid $D=E(a_1, \dots ,a_n)$. We follow an identical line of argument to the proof of Theorem \ref{stable} above, letting $E(b_1, \dots ,b_n)$ be the filling ellipsoid for $(M,\omega)$ from Theorem \ref{sec4}.
\begin{lemma} There exists a full filling \[\sqcup_k E \big(a_1, \dots ,a_n \big) \hookrightarrow E\big( a_1, \dots ,a_{n-1}, ka_n \big)\] for any $k \in \NN$.
\end{lemma}

\begin{proof} Define two subsets of $\RR^{2n}$ as follows. $$T=\{0\le x_1, \dots ,x_n, 0 \le y_1, \dots ,y_n \le 1|\sum_{i=1}^n \frac{x_i}{a_i} \le 1\}$$ and
$$U=\{0\le x_1, \dots ,x_n, 0 \le y_1, \dots ,y_{n-1} \le 1, 0 \le y_n \le k|\sum_{i=1}^n \frac{x_i}{a_i} \le 1\}.$$
Lemma $5.3.1 (i)$ from \cite{Sh1} gives an embedding $E(a_1, \dots ,a_n) \hookrightarrow T$ and hence an embedding \[\sqcup_k E \big(a_1, \dots ,a_n \big) \hookrightarrow U.\] Lemma $5.3.1 (ii)$ implies that there exists an embedding $U \hookrightarrow E(a_1, \dots ,ka_n)$. Composing these two our lemma follows.
\end{proof}

But by Theorem \ref{main} there exists a constant $S(b_1,\dots ,b_n)$ such that there exists a full filling $ \gamma E(1,\dots ,1,k) \hookrightarrow E(b_1, \dots ,b_n)$ for any $k \ge S$. This map and the embedding from our lemma together give packing stability for $(M,\omega)$ with respect to $D$. $\qed$

\vspace{0.2in}

Now we proceed to prove Theorem \ref{sec4} by induction on the dimension of the manifold $M$. Observe that Moser's theorem implies a surface can be fully filled by any number of balls, so we assume that the theorem holds for all manifolds of dimension $2(n-1)$ and let $(M,\omega)$ be rational of dimension $2n$.

Using the rationality assumption, a theorem of S. K. Donaldson, \cite{don}, says that $(M,\omega)$ can be polarized in the sense that there exists a symplectic codimension $2$ submanifold $N$ which is Poincar\'{e} dual to $m[\omega]$ for a suitable multiple $m$. Let $\tau=\omega|_N$.

Let $SDB(N,\tau,m)$ be a standard symplectic disk bundle over $N$. This is a symplectic manifold whose underlying smooth manifold is the disk bundle over $N$ with Euler class $[m\tau] \in H^2(N,\ZZ)$ and which has a symplectic form restricting to $\tau$ on $N$ and with fibers of area $\frac{1}{m}$. As a symplectic manifold, $SDB(N,\tau,m)$ is well defined up to isotopy. For more details see \cite{bir0}, section $2$, or \cite{ops2}, section $1$. Given the above, we can now state the Biran decomposition.

\begin{theorem} [Theorem $1$, \cite{ops2},\cite{bir0}] \label{opst} There exists a full filling $SDB(N,\tau,m) \hookrightarrow (M,\omega)$.
\end{theorem}

By our induction hypothesis, there exits $b_1, \dots ,b_{n-1}$ and a full filling $E(b_1,\dots ,b_{n-1}) \hookrightarrow (N,\tau)$. Now we apply the construction of Opshtein.

\begin{lemma} [Lemma $2.1$, \cite{opshtein}] \label{opss} A symplectic embedding $E(b_1,\dots ,b_{n-1}) \hookrightarrow (N,\tau)$ can be extended to a symplectic embedding $E(b_1,\dots ,b_{n-1},\frac{1}{m}) \hookrightarrow SDB(N,\tau,m)$.
\end{lemma}

If $E(b_1,\dots ,b_{n-1}) \hookrightarrow (N,\tau)$ is a full filling then the embedding $E(b_1,\dots ,b_{n-1},\frac{1}{m}) \hookrightarrow SDB(N,\tau,m)$ of Lemma \ref{opss} is also a full filling. Thus, combining Theorem \ref{opst} and Lemma \ref{opss} we have a full filling $E(b_1,\dots ,b_{n-1},\frac{1}{m}) \hookrightarrow (M,\omega)$ and we conclude the proof of Theorem \ref{sec4} by induction.

\section{Examples}

\subsection{$\CC P^n$}

Complex projective space $\CC P^n$ is fully filled by a ball. Recall also that $E(1,\dots ,1,k)$ can be fully filled by $k$ balls, see \cite{busehind}, Lemma $4.1$. Therefore we can establish our stability bound for $\CC P^n$ by producing an embedding $E(1,\dots ,1,k) \hookrightarrow B(k^{\frac{1}{n}})$ for all $k \ge (8\frac{1}{36})^{\frac{n}{2}}$.

We consider the following sequence of $n-1$ potential embeddings.

$$E(1,\dots ,1,k) \hookrightarrow E(k^{\frac{1}{n}},1,\dots ,1,k^{\frac{n-1}{n}})$$
$$\hookrightarrow E(k^{\frac{1}{n}},k^{\frac{1}{n}},1,\dots ,1,k^{\frac{n-2}{n}}) \hookrightarrow \dots \hookrightarrow E(k^{\frac{1}{n}},\dots ,k^{\frac{1}{n}},1,k^{\frac{2}{n}})$$
$$\hookrightarrow E(k^{\frac{1}{n}},\dots ,k^{\frac{1}{n}})=B(k^{\frac{1}{n}}).$$

The first $n-2$ embeddings will exist by Theorems \ref{one} and Proposition \ref{two} provided we have
\begin{equation}
k^{\frac{n-i}{n}} \ge \frac{(5k^{\frac{n-i-2}{n}}+ 16)^2}{16k^{\frac{n-i-2}{n}}} = \frac{25}{16}k^{\frac{n-i-2}{n}} + 10 + 16k^{\frac{-(n-i-2)}{n}}
\end{equation}
for all $0 \le i \le n-3$.

For the final embedding, as the target is a ball, we get an improved bound by appealing directly to \cite{mcdsch}, Theorem $1.1.2$ (iv). This says that $E(1,k^{\frac{2}{n}}) \hookrightarrow B(k^{\frac{1}{n}})$ whenever $k^{\frac{2}{n}} \ge 8\frac{1}{36}$. Therefore by Proposition \ref{two} the final embedding will also exist provided $k^{\frac{2}{n}} \ge 8\frac{1}{36}$.

For the first $n-2$ embeddings we require

\begin{equation}\label{kineq}
\frac{25}{16}k^{\frac{-2}{n}} + 10 k^{\frac{-(n-i)}{n}} + 16 k^{\frac{-2(n-i-1)}{n}} \le 1
\end{equation}

 for all $0 \le i \le n-3$. The left hand side is an increasing function of $i$, so it suffices that $$\frac{25}{16}k^{\frac{-2}{n}} + 10 k^{\frac{-3}{n}} + 16 k^{\frac{-4}{n}} \le 1.$$
A calculation shows that this also holds when $k^{\frac{2}{n}} \ge 8\frac{1}{36}$ and so Theorem \ref{bound} $(i)$ is proved. \qed

There is in fact a better bound valid in dimension $6$.

\begin{proposition}
$N_{{\rm stab}}(\CC P^3) \le 21$
\end{proposition}

\begin{proof}
Part $(i)$ of Theorem \ref{bound} already implies that $N_{{\rm stab}}(\CC P^3) \le 23.$
Thus we just need to check that the same embeddings exist when $k=21$ and $k=22$. Direct computation shows that the inequality \ref{kineq} continues to hold for these $k$, so we are reduced to showing $E(1,k^{\frac{2}{3}}) \hookrightarrow B(k^{\frac{1}{3}})$ when $k=21, 22$.
We appeal again to McDuff-Schlenk's results from \cite{mcdsch}, Theorem $5.2.3$ and its corresponding Table $5.2$. This says that $f_1(\alpha)=\sqrt{\alpha}$ on the interval $[ 7 \frac{1}{9}, 8 ]$ with the exception of eight intervals that are explicitly provided in their statement. The numbers $k^{\frac{2}{3}}$ for $k=21,22$ do not lie in these intervals, thus this embedding also exists and hence the whole sequence of embeddings provided in the proof of Theorem \ref{bound}.
\end{proof}

\subsection{$H^n_d$}

We first make the following observation.

{\it Fact.} $H^n_d$ embeds as a polarizing hypersurface in $H^{n+1}_d$ Poincar\'{e} dual to the symplectic form.

Indeed, $H^n_d = H^{n+1}_d \cap \CC P^{n+1}$ where $H^{n+1}_d \subset \CC P^{n+2}$ and we think of $\CC P^{n+1}$ as a hyperplane intersecting $H^{n+1}_d$ transversely.

Now, $H^1_d$ is a curve of degree $d$ and so is fully filled by $B^2(d)$, normalizing the Fubini-Study form on the $\CC P^n$ so that lines have area $1$. Hence, combining the above fact with Lemma \ref{opss} and arguing by induction we see that $E(1,\dots ,1,d)$ fully fills $H^n_d$.

Therefore, arguing as in case $(i)$ by applying Lemma $4.1$ of \cite{busehind}, to demonstrate Theorem \ref{bound} $(ii)$ it suffices to show that $$E(1,\dots ,1,k) \hookrightarrow \frac{k^{\frac{1}{n}}}{d^{\frac{1}{n}}}E(1,\dots ,1,d)$$ for all integers $k \ge (\frac{25}{16}d+10d^{-\frac{(n-2)}{n}}+16d^{-\frac{2(n-1)}{n}})^{\frac{n}{2}}$.

We now consider the sequence of $(n-1)$ potential embeddings
$$E(1,\dots ,1,k) \hookrightarrow E(\frac{k^{\frac{1}{n}}}{d^{\frac{1}{n}}}, 1, \dots, 1, k^{\frac{n-1}{n}}d^{\frac{1}{n}}) \hookrightarrow$$
$$E(\frac{k^{\frac{1}{n}}}{d^{\frac{1}{n}}}, \frac{k^{\frac{1}{n}}}{d^{\frac{1}{n}}}, 1, \dots, 1, k^{\frac{n-2}{n}}d^{\frac{2}{n}}) \hookrightarrow \dots \hookrightarrow E(\frac{k^{\frac{1}{n}}}{d^{\frac{1}{n}}}, \dots, \frac{k^{\frac{1}{n}}}{d^{\frac{1}{n}}}, k^{\frac{1}{n}}d^{\frac{n-1}{n}})$$ $$=\frac{k^{\frac{1}{n}}}{d^{\frac{1}{n}}}E(1,\dots ,1,d).$$

By Proposition \ref{two} the $(i+1)$st embedding exists provided there exists an embedding $$E(1,k^{\frac{n-i}{n}}d^{\frac{i}{n}}) \hookrightarrow E(\frac{k^{\frac{1}{n}}}{d^{\frac{1}{n}}}, k^{\frac{n-i-1}{n}}d^{\frac{i+1}{n}})$$ which by Theorem \ref{one} exists provided
$$k^{\frac{n-i}{n}}d^{\frac{i}{n}} \ge \frac{25}{16}k^{\frac{n-i-2}{n}}d^{\frac{i+2}{n}}+ 10 + \frac{16}{k^{\frac{n-i-2}{n}}d^{\frac{i+2}{n}}}$$
or equivalently
$$\frac{25}{16}k^{\frac{-2}{n}}d^{\frac{i+2}{n}} + 10k^{-\frac{(n-i)}{n}}d^{\frac{-i}{n}} + 16 k^{-2\frac{(n-i-1)}{n}}d^{-2\frac{(i+1)}{n}} \le 1.$$
For $k \ge d$ and $n$ fixed, the left hand side of this last inequality is an increasing function of $i$. Therefore it suffices to check it when $i=n-2$. In this case we get
$$\frac{25}{16}k^{\frac{-2}{n}}d+ 10k^{\frac{-2}{n}}d^{-\frac{(n-2)}{n}} + 16k^{\frac{-2}{n}}d^{-2\frac{(n-1)}{n}} \le 1$$ which is equivalent to $$k^{\frac{2}{n}} \ge \frac{25}{16}d+10d^{-\frac{(n-2)}{n}}+16d^{-\frac{2(n-1)}{n}}$$ and so holds in our range as required.

\begin{remark}
Still following our methods here as well as those from \cite{busehind}, the estimates for the upper bound on the stability numbers presented here may be improved if we can improve,  by a much more complicated number theoretical investigation of the ECH vectors,  the four dimensional embedding result from Theorem \ref{oneemb}. However, in the case when the target manifold is $\CC P^3$ the inequality $N_{stab}  \leq 21$ cannot be improved in this way as it relies at the last step on a sharp estimate from McDuff- Schlenk  giving the precise range for volume filling ellipsoids in the unit ball. It remains a very interesting question to find methods to obstruct symplectic volume fillings  of ellipsoids or disjoint balls into the unit six dimensional ball. As will be shown in \cite{busehindprep}, it appears that obstructions arising from symplectic field theory that have been succesful in \cite{hindkerman} when targets were cylinders fail here.
\end{remark}

\end{document}